\documentclass[10pt]{article}
\font\smallit=cmti10

\usepackage{amssymb,latexsym,amsmath,epsfig,amsthm} 

\makeatletter

\renewcommand\section{\@startsection {section}{1}{\z@}
{-30pt \@plus -1ex \@minus -.2ex}
{2.3ex \@plus.2ex}
{\normalfont\normalsize\bfseries}}

\renewcommand\subsection{\@startsection{subsection}{2}{\z@}
{-3.25ex\@plus -1ex \@minus -.2ex}
{1.5ex \@plus .2ex}
{\normalfont\normalsize\bfseries}}

\renewcommand{\@seccntformat}[1]{\csname the#1\endcsname. }

\makeatother

\begin{document}

\newtheorem{theorem}{Theorem}[section]
\newtheorem{lemma}[theorem]{Lemma}
\newtheorem{definition}[theorem]{Definition}
\newtheorem{remark}[theorem]{Remark}
\newtheorem{claim}[theorem]{Claim}
\newtheorem{example}[theorem]{Example}
\newtheorem{algorithm}[theorem]{Algorithm}
\newtheorem{corollary}[theorem]{Corollary}
\newtheorem{observation}[theorem]{Observation}
\newtheorem{conjecture}[theorem]{Conjecture}
\newtheorem{proposition}[theorem]{Proposition}

\newcommand{\N}{{\mathbb{N}}}
\newcommand{\Z}{{\mathbb{Z}}}
\newcommand{\F}{{\mathbb{F}}}
\newcommand{\E}{{\mathbb{E}}}

\newcommand{\D}[1]{{\mathbb#1}}
\newcommand{\PP}{{\D{P}}}
\newcommand{\ZZ}{{\D{Z}}}
\newcommand{\QQ}{{\D{Q}}}
\newcommand{\RR}{{\D{R}}}
\newcommand{\CC}{{\D{C}}}
\newcommand{\NN}{{\D{N}}}
\newcommand{\SSS}{{\D{S}}}
\newcommand{\EE}{{\D{E}}}
\newcommand\aplus{\stackrel{\alpha}{+}}
\newcommand\adot{\stackrel{\alpha}{\cdot}}
\newcommand\gdot{\stackrel{\gamma}{\cdot}}
\newcommand\Aut{\mathrm{Aut}}
\newcommand\Inn{\mathrm{Inn}}
\newcommand\Out{\mathrm{Out}}
\newcommand\Sym{\mathrm{Sym}}
\newcommand\id{\mathrm{id}}  
\newcommand\chr{\mathrm{char}}
\newcommand\harrow{\stackrel{h}{\rightarrow}}
\newcommand\hiarrow{\stackrel{(h^\wedge)}{\rightarrow}}

\newcommand\ontop[3]{\mathrel{\raise#3\rlap{$\scriptstyle #1$}\mathord{#2}}}
\newcommand\gplus{\ontop{\mkern3mu\gamma}{+}{8pt}}
\newcommand\tdot{\ontop{\mkern-1mu\theta}{\cdot}{6pt}}
\newcommand\tpdot{\ontop{\mkern-2mu\theta'}{\cdot}{6pt}}
\newcommand\tidot{\ontop{\theta^{\!-\!1}}{\,\cdot\,}{6pt}}

\renewcommand{\thefootnote}{\fnsymbol{footnote}}

\begin{center}
\uppercase{\bf On the Inverse Erd\H{o}s-Heilbronn Problem for Restricted Set Addition in Finite Groups}
\vskip 20pt

{\bf Suren M. Jayasuriya\footnote[2]{\baselineskip=12pt This work was started while S.M. Jayasuriya and S.D. Reich were undergraduates at the University of Pittsburgh in a directed study course supervised by Dr. Jeffrey P. Wheeler.}}\\
{\smallit Department of Electrical and Computer Engineering, Cornell University, Ithaca, New York 14853, USA}\\
{ \tt sj498@cornell.edu}\\
\vskip 10pt

{\bf Steven D. Reich\footnotemark[2]} \\
{\smallit Department of Mathematics, The University of Pittsburgh, Pittsburgh, Pennsylvania, 15260, USA}\\
{ \tt sdr32@pitt.edu}
\vskip 10pt

{\bf Jeffrey P. Wheeler} \\
{\smallit Department of Mathematics, The University of Pittsburgh, Pittsburgh, Pennsylvania, 15260, USA}\\
{ \tt jwheeler@pitt.edu}
\end{center}
\vskip 30pt

\vskip 30pt

\renewcommand{\thefootnote}{}


\renewcommand{\thefootnote}{\arabic{footnote}}
\setcounter{footnote}{0}

\centerline{\bf Abstract}
\noindent
We provide a survey of results concerning both the direct and inverse problems to the Cauchy-Davenport theorem and Erd\H{o}s-Heilbronn problem in Additive Combinatorics. We formulate a conjecture concerning the inverse Erd\H{o}s-Heilbronn problem in nonabelian groups. We prove an inverse to the Dias da Silva-Hamidoune Theorem to $\Z/n\Z$ where $n$ is composite, and we generalize this result for nonabelian groups. 


\section{Introduction}

A basic object in additive combinatorics/additive number theory is the sumset of sets $A$ and $B$:

\begin{definition}\index{sumset}[Sumset]\label{definition:A+B}\ \\
$$A+B := \{ a+b \mid a \in A, b \in B \}.$$
\end{definition}

A simple example of a problem in Additive Number Theory is given two subsets $A$ and $B$ of a set of integers, what facts can we determine about sumset $A+B$? One such classic problem was a conjecture of Paul Erd\H{o}s and Hans Heilbronn \cite{Erdos Col. Con.}, an open problem for over $30$ years until proved in $1994$.  The conjecture originates from a  theorem proved by Cauchy \cite{Cauchy} in $1813$ and independently
by Davenport \cite{Davenport} in $1935$.

In this paper, we present a survey of results concerning both the Cauchy-Davenport theorem and Erd\H{o}s-Heilbronn problem\footnote{\baselineskip=12pt See Appendix for a timeline summarizing these results}. We introduce the two main types of problems: direct and inverse problems, and we then list results and extensions of these theorems into groups. In particular, we formulate a conjecture concerning the inverse Erd\H{o}s-Heilbronn problem into nonabelian groups and provide a nontrivial example to support it. In section 5, we present an elementary proof providing an inverse to the Dias da Silva-Hamidoune theorem for $\Z/n\Z$ where $n$ is composite. In section 6, we generalize this result to nonabelian groups which proves one direction of the conjecture under an assumption that our sets are arithmetic progressions with the same common difference. This result is a first step towards proving the full inverse Erd\H{o}s-Heilbronn problem in nonabelian groups.

\section{The Cauchy-Davenport Theorem and Erd\H{o}s-Heilbronn Problem}

As described by Melvyn B. Nathanson in \cite{Nath}, a {\em direct
problem}\index{direct problem} in Additive Number Theory is a
problem concerned with properties of the resulting sumset.  We first consider two direct results; the first a classic result and the second a simple adaptation of the first - yet significantly more subtle to prove.

\subsection{The Cauchy-Davenport Theorem}\label{section:C-D Theorem}

The first result is a theorem proved by
Cauchy\footnote{\baselineskip=12pt Cauchy used this theorem to prove
that $Ax^2 + By^2 + C \equiv 0 (\!\!\!\!\mod p)$ has solutions
provided that $ABC \not \equiv 0$.  Cauchy then used this to provide a new proof of a lemma Lagrange used to establish his four squares theorem in 1770 \cite{Alon - CombNull}.} in $1813$
\cite{Cauchy} and independently by Davenport in $1935$
\cite{Davenport} (Davenport discovered in $1947$
\cite{DavenportHist} that Cauchy had previously proved the theorem).
In particular,

\begin{theorem}\label{theorem:cauchy-davenport}[Cauchy-Davenport]\index{Cauchy-Davenport Theorem}\ \\
Let $A$ and $B$ be nonempty subsets of $\ZZ/p\ZZ$ with $p$ prime.
Then $|A+B| \geq \min \{p, |A|+|B|-1\}$ where $A+B := \{a+b\mid a
\in A \text{ and } b \in B\}.$
\end{theorem}

We note that in $1935$ Inder Cholwa \cite{Chowla} extended the
result to composite moduli $m$ when $0 \in B$ and the other members
of $B$ are relatively prime to $m$.

\subsection{A Conjecture of Erd\H{o}s-Heilbronn}

The second result we consider is a slight modification of the Cauchy-Davenport Theorem and is surprisingly much more difficult. In the early $1960$'s, Paul Erd\H{o}s and Hans Heilbronn conjectured that if the sumset addition is restricted to only distinct elements, then the lower bound is reduced by two.
Erd\H{o}s stated this conjecture during a 1963 University of Colorado number theory
conference \cite{Erdos Col. Con.}. While the conjecture did not appear in their $1964$ paper on sums of sets of congruence
classes \cite{Erdos Heilbronn},  Erd\H{o}s lectured on the conjecture (see \cite{Nath}, p.106). The conjecture was formally stated in \cite{Erdos} and \cite{Erdos Graham} as follows: 

\begin{theorem}\label{problem:erdos-heilbronn}[Erd\H{o}s-Heilbronn Problem]\label{theorem:EHP}
\index{The Conjecture of Erd\H{o}s-Heilbronn}\ \\
Let $p$ be a prime and $A$, $B \subseteq \ZZ/p\ZZ$ with $A,B \neq
\emptyset$. Then $$|A \dot{+} B| \geq \min \{
p, |A|+|B|-3 \},$$ where $A \dot{+} B := \{ a + b \text{ mod p } \mid
a \in A$, $b \in B$ and $a \neq b$ \}.
\end{theorem}

The conjecture was first proved for the case $A=B$ by J.A. Dias da
Silva and Y.O. Hamidoune in $1994$~\cite{Dias da Silva} using
methods from linear algebra with the more general case (namely $A \neq B$) established
by Noga Alon, Melvin B. Nathanson, and Imre Z. Ruzsa using their powerful
polynomial method in $1995$~\cite{Alon}.
\begin{remark}\rm
Throughout this paper, we will label the Erd\H{o}s-Heilbronn problem for the case where $A=B$ as the Dias da Silva-Hamidoune theorem.
\end{remark}


\section{Extension of the Problems to Groups}

The structures over which the Cauchy-Davenport Theorem holds have
been extended beyond $\Z/p\Z$.  Before stating the
extended versions, the following definition is needed.

\begin{definition}[Minimal Torsion Element]
\label{definition:minimal torsion element}
Let\/ $G$ be a group.  We
define $p(G)$ to be the smallest positive integer $p$ for which
there exists a nontrivial element\/ $g$ of\/ $G$ with $pg=0$ (or, if
multiplicative notation is used, $g^{p}=1$). If no such $p$ exists, we
write $p(G)=\infty$.
\end{definition}

\begin{remark}\label{remark:p(G) as a prime factor}\rm
When $G$ is finite, then $p(G)$ is the smallest prime factor of $|G|$ or equivalently, $p(G)$ is the size of the smallest nontrivial subgroup of $G$. 
\end{remark}

Equipped with this we can state the Cauchy-Davenport Theorem which was extended to abelian groups by Kneser~\cite{Kneser} and then to all finite
groups by Gy. K{\'a}rolyi~\cite{Gyula3}:

\begin{theorem}[Cauchy-Davenport Theorem for Finite Groups]\label{Finite Group C-D}
If\/ $A$ and\/ $B$ are non-empty subsets of a finite group $G$, then
$|AB|\ge\min\{p(G),|A|+|B|-1\}$, where
$AB:=\{a\cdot b\mid a\in A\text{ and\/ }b\in B\}$.
\end{theorem}

 J.E. Olson~\cite{Olson} also proved the following result which implies the Cauchy-Davenport theorem for arbitrary groups:
\begin{theorem} If \/ $A$ and \/ $B$ are two finite sets in a group, then both 
\\
$1)$ \hspace{0.4cm} $|A\cdot B| \geq |A| + \frac{1}{2}|B|$ unless $ A\cdot B \cdot (-B\cdot B) = A \cdot B$, and 
\\
$2)$ \hspace{0.4cm} there is a subset $S \subseteq A\cdot B$ and subgroup $H$ such that $|S| \geq |A| + |B| - |H|$ 
\\
\indent \hspace{0.4cm} and either $H\cdot S = S$ or $ S\cdot H = S$.
\end{theorem}

Similarly, work has been done to extend the Erd\H{o}s-Heilbronn Problem into groups. Starting with abelian groups, Gy. K{\'a}rolyi
proved the following result in  ~\cite{Gyula1, Gyula2}:

\begin{theorem} If $A$ is a nonempty subset of an abelian group $G$, then $|A \dot{+} A| \geq \min\{ p(G), 2|A|-3\}.$
\end{theorem}

He also extended the Erd\H{o}s-Heilbronn Problem to cyclic groups of prime powered order in 2005~\cite{Gyula Compact}. 

To state the result that extends the problem into finite, not necessarily abelian, groups, we introduce the following definition:

\begin{definition}\label{definition:tdot}
For a group $G$, let\/ $\Aut(G)$ be the group of automorphisms of\/~$G$.
Suppose $\theta\in\Aut(G)$ and\/ $A,B\subseteq G$. Write
\[
 A\tdot B:=\{a\cdot\theta(b)\mid a\in A,\,b\in B,\text{ and }a\ne b\,\}.
\]
\end{definition}

Given this definition, we can clearly state the Erd\H{o}s-Heilbronn theorem for finite groups which was proven in 2009 by P. Balister and J.P. Wheeler ~\cite{Balister/Wheeler}:

\begin{theorem}[Generalized Erd\H{o}s-Heilbronn Problem for Finite Groups]\label{theorem:GEH}
Let $G$ be a finite group, $\theta \in \Aut(G)$, and let $A,B
\subseteq G$ with $|A|$, $|B| > 0$.  Then $|A \tdot B| \geq \min
\{p(G)-\delta_{\theta}, |A|+|B|-3 \}$ where $A \tdot B := \{ a \cdot
\theta(b) \mid a \in A, b \in B,$ and $a \neq b \}$ and where \[
\delta_{\theta} = \left \{
\begin{array}{rll}
1 & \mbox{if} & \theta \mbox{ has even order in } \Aut(G),\\
0 & \mbox{if} & \theta \mbox{ has odd order in } \Aut(G).
\end{array}
\right.
\]
\end{theorem}
We note that Lev \cite{Lev} has shown that Theorem \ref{theorem:GEH} does not hold in general for an arbitrary bijection $\theta$.

\section{Inverse Problems}

The previous problems were concerned with properties of the sumset
given some knowledge of the individual sets making up the sumset.
This leads one to consider questions in the other direction. In
particular, if we know something about the sumset, does this give us
any information about the individual sets making up the sumset?
Again we borrow the language of \cite{Nath} and refer to these
problems as {\em inverse problems}\index{inverse problem}.  In each
case, the inverse of the previously stated problems yields beautiful
results.

\subsection{Inverse Problems for the
Cauchy-Davenport Theorem}

The first result is the inverse to the Cauchy-Davenport theorem in $\ZZ / p \ZZ$ due to A.G. Vosper \cite{Vosper}:

\begin{theorem}[Vosper's Inverse Theorem]\index{Vosper's Theorem}\ \\
Let $A$ and $B$ be finite nonempty subsets of $\ZZ / p \ZZ$. Then $|A+B| = |A| + |B| -1$ if and only if one of the following cases holds: 
\begin{itemize}
\item[] $(i)$ $|A| =1$ or $|B| = 1$;
\item[] $(ii)$ $A + B = \ZZ / p \ZZ$;
\item[] $(iii)$ $|A+B| = p-1$ and $B$ is the complement of the set $c - A$ in $\ZZ / p \ZZ$ where $\{ c\} = \left( \ZZ / p \ZZ \right)\setminus  \left( A+B \right)$;
\item[] $(iv)$ A and B are arithmetic progressions of the same common difference.
\end{itemize}
\end{theorem}

\begin{remark}\rm
We note that when $|A+B| = |A| + |B| - 1$, we label sets $A$ and $B$ as a \textit{critical pair}.
\end{remark} 
In $1960$, J.H.B. Kemperman~\cite{Kemperman} extended a weaker
version of Vosper's Theorem to abelian groups.  Namely,

\begin{theorem}[Kemperman]\ \\
Let $A$ be a nonempty subset of an abelian group.  Let $p(G)$ be as
in Definition \ref{definition:minimal torsion element}.  Suppose
$p(G) > 2|A|- 1$.  Then $|A+A| = 2|A|- 1$ if and only if $A$ is an
arithmetic progression.
\end{theorem}

To generalize the inverse problem in groups, we introduce the notion of an arithmetic progression in a group:

\begin{definition}[Group Arithmetic Progression]\index{group
arithmetic progression}\ \\
Let $G$ be a group and $A \subseteq G$ with $|A|=k$.  Then $A$ is a
{\em group arithmetic progression} if there exists both $g$ and $h$ in
$G$ such that $A = \{ g + ih \mid 1 \leq i < k \}$.
\end{definition}

In the above definition, we say that $A$ is a $k$-term group
arithmetic progression with common difference $h$. If the group $G$ is nonabelian, we utilize multiplicative notation and form the intuitive definitions of left and right arithmetic progressions. 

Y.O. Hamidoune further extended this idea to finitely generated
groups \cite{Hamidoune}.

\begin{theorem}[Hamidoune]\ \\
Let $G$ be a (not necessarily abelian) group generated by a finite
subset $S$ where $0 \in S$.  Then either
\begin{enumerate}
\item for every subset $T$ such that $2 \leq |T| < \infty$, we have
$|S+T| \geq \min \{|G|-1, |S|+|T| \}$ or
\item $S$ is an arithmetic progression.
\end{enumerate}
\end{theorem}

Finally, Gy. K{\'a}rolyi obtained the inverse to the Cauchy-Davenport theorem in finite groups \cite{Gyula3}:

\begin{theorem}
Let $A,B$ be subsets of a finite group $G$ such that $|A| = k$, $|B|=l$, and $k+l -1 \leq p(G) - 1.$ Then $|AB| = k +l - 1$ where $AB = \{ ab | \hspace{0.2cm} a \in A, b \in B\}$ if and only if one of the following conditions holds:
\begin{itemize}
\item[] $(i)$ $k=1$ or $l=1$;
\item[] $(ii)$ there exists $a,b,q \in G$ such that 
$$ A = \{a, aq, aq^{2}, \dots, aq^{k-1}\} \text{    and    }  B = \{ b, qb, q^{2}b, \dots, q^{l-1}b\};$$
\item[] $(iii)$ $k+l-1 = p(G)-1$ and there exists a subgroup $F$ of $G$ of order $p(G)$ and elements $u, v \in G, z \in F$ such that 
$$ A \subset uF, B \subset Fv \text{   and   } A = u(F\setminus zvB^{-1}).$$
\end{itemize}
\end{theorem}

\subsection{Inverse Problems Related to the Erd\H{o}s-\\
Heilbronn Problem}

We as well have inverse problems for the Erd\H{o}s-Heilbronn
Problem.  More specifically, we have inverse results for the Dias da Silva-Hamidoune Theorem. In particular, Gy. K{\'a}rolyi established the following in
\cite{Gyula Compact}:

\begin{theorem}[Inverse Theorem of the Dias da Silva-Hamidoune Theorem]\label{invDH}
Let $A$ be a subset of $\ZZ / p \ZZ$ where $p$ is a prime.  Further
suppose $|A| \geq 5$ and $p > 2|A| -3$.  Then $|A \dot{+} A| =
2|A|-3$ if and only if $A$ is an arithmetic progression.
\end{theorem}

He also extended this result to the following:
\begin{theorem} \label{K2005}\ \\
Let $A$ be a subset of an abelian group $G$ where $p(G)$ is as in
Definition \ref{definition:minimal torsion element} prime. Further
suppose $|A| \geq 5$ and $p(G) > 2|A| -3$.  Then $|A \dot{+} A| =
2|A|-3$ if and only if $A$ is a group arithmetic progression.
\end{theorem}

Similar to the inverse to the Cauchy-Davenport theorem which holds in nonabelian groups, Gy. K{\'a}rolyi conjectured that the inverse to the Erd\H{o}s-Heilbronn problem in a nonabelian setting should hold, namely that only sets that are arithmetic progressions achieve the lower bound placed on their restricted sumset by the Erd\H{o}s-Heilbronn problem \cite{Gyula09}. 
We note that previous work by V. F. Lev proved an inverse Erd\H{o}s-Heilbronn theorem in an asymptotic sense for $\Z/p\Z$ with $p$ very large \cite{Lev1}, and that this result was improved by Van Vu and Philip M. Wood \cite{Vu}. Du and Pan \cite{Du/Pan} have recently submitted a proof for the following result:

\begin{theorem}
Suppose that $A,B$ are two non-empty subsets of the finite nilpotent group $G$. If $A \neq B$, then the cardinality of $A  \overset{\iota}{\cdot} B$ is at least the minimum of $p(G)$ and $|A|+|B|-2$. 
\end{theorem}

They also proved that if $|A \dot{+} A| = 2|A|-3$ with $A$ a non-empty subset of a finite group G with $|A| < (p(G) +3)/2$, then $A$ is commutative. 

Thus we formulate a conjecture for the inverse Erd\H{o}s-Heilbronn problem to hold in groups that are not nilpotent. We only consider restricted product sets with $\theta = \iota$ being the identity automorphism. 

\begin{conjecture}\label{IEHconj}
Let $A,B$ be nonempty subsets of a finite (not necessarily abelian), non-nilpotent group $G$ where $p(G)$ is as in Definition \ref{definition:minimal torsion
element}.  Further suppose $|A| = k \geq 3$, $|B| = l \geq 3$, and $k+l - 3 < p(G) $. Then $|A \overset{\iota}{\cdot} B| = |A| + |B| -3$ where $A\overset{\iota}{\cdot} B = \{ ab \hspace{0.1cm} \vert \hspace{0.1cm} a\in A, b\in B, a\neq b\}$ if and only if there exists $a,q \in G$ such that
$$ A = \{ a, aq, aq^2, \dots, aq^{k-1}\} \text{   and   } B = \{a, qa, q^2 a, \dots , q^{l-1}a\}$$ where $aq^{k-1} = q^{l-1}a$, i.e. $A,B$ share the same endpoints. 
\end{conjecture}

We note that the if statement of the conjecture is trivial, and if $G$ is nilpotent, such pairs $A$ and $B$ only exist when $A=B$ is a progression lying in an abelian subgroup as shown in \cite{Du/Pan}. 

We now present an example of sumset addition in a non-nilpotent group that gives evidence that extra critical pairs can indeed arise. The difficulty in testing this conjecture is finding an appropriate group whose $p(G)$ is relatively large compared to the cardinalities of sets $A$ and $B$. Standard nonabelian groups such as dihedral groups or symmetric groups do not satisfy this condition because $p(G) = 2$ in these groups. Thus in the following example, we construct a large (in terms of $p(G)$) nonnilpotent group to test the conjecture. Note: since writing this manuscript, we have discovered simpler and more general examples that will be presented in future work \cite{JKRW}.

\begin{example}\label{iehexample}
\end{example}
Let $G$ be a nonabelian group constructed by $G = \left( \mathbb{Z}_{47} \times \mathbb{Z}_{47} \right) \rtimes_{\phi} \mathbb{Z}_{23}$.\footnote{\baselineskip=12pt We denote $\Z/47\Z$ as $\mathbb{Z}_{47}$ for notational purposes, and note that $\mathbb{Z}_{47} \cong \mathbb{F}_{47}^{+}$.} Since $Aut(\mathbb{Z}_{47} \times \mathbb{Z}_{47}) \cong GL_{2} (\mathbb{F}_{47})$, we construct the homomorphism $\phi:\mathbb{Z}_{23} \to GL_{2} (\mathbb{F}_{47})$ as follows:

$$\phi(x)=  \left( \begin{array}{cc}
2^{x} & 0  \\
0 & 1 \end{array} \right).$$

Explicitly, we can think of elements of G having the form $ \left( \left(\begin{array}{cc} x \\ y \end{array}\right) , z \right)$ where the group operation is
$$
\left( \left(\begin{array}{cc} x \\ y \end{array}\right) , z \right) \cdot_{G}  \left( \left(\begin{array}{cc} x' \\ y' \end{array}\right) , z' \right)  =  \left( \left(\begin{array}{cc} x \\ y \end{array}\right) + \phi(z) \left(\begin{array}{cc} x' \\ y' \end{array}\right), z+z' \right).$$
 $$
= \left( \left( \begin{array}{cc} x+2^{z}x' \\ y + y' \end{array}\right), z+z' \right).$$

Observe that $p(G) = 23$ since $|G| = 23 \cdot 47^2 $. Take $$A = \left\{   \left( \left(\begin{array}{cc} 0 \\ 0 \end{array}\right) , 1 \right) \cdot \left( \left(\begin{array}{cc} 1 \\ 0 \end{array}\right) , 0 \right)^{k} \bigm| 0 \leq k \leq 4 \right\}$$
$$=  \left\{ 
\left( \left(\begin{array}{cc} 0 \\ 0 \end{array}\right) , 1 \right),
\left( \left(\begin{array}{cc} 2 \\ 0 \end{array}\right) , 1 \right),
\left( \left(\begin{array}{cc} 4 \\ 0 \end{array}\right) , 1 \right),
\left( \left(\begin{array}{cc} 6 \\ 0 \end{array}\right) , 1 \right), 
\left( \left(\begin{array}{cc} 8 \\ 0 \end{array}\right) , 1 \right) \right\}.$$

and let $$ B = \left\{  \left( \left(\begin{array}{cc} 1 \\ 0 \end{array}\right) , 0 \right)^{l} \cdot \left( \left(\begin{array}{cc} 0 \\ 0 \end{array}\right) , 1 \right)  \bigm| 0 \leq l \leq 8 \right\}$$
$$=  \left\{  
\left( \left(\begin{array}{cc} 0 \\ 0 \end{array}\right) , 1 \right),
\left( \left(\begin{array}{cc} 1 \\ 0 \end{array}\right) , 1 \right),
\left( \left(\begin{array}{cc} 2 \\ 0 \end{array}\right) , 1 \right),
\left( \left(\begin{array}{cc} 3 \\ 0 \end{array}\right) , 1 \right),
\left( \left(\begin{array}{cc} 4 \\ 0 \end{array}\right) , 1 \right),\right.$$
$$ \left. 
\left( \left(\begin{array}{cc} 5 \\ 0 \end{array}\right) , 1 \right),
\left( \left(\begin{array}{cc} 6 \\ 0 \end{array}\right) , 1 \right),
\left( \left(\begin{array}{cc} 7 \\ 0 \end{array}\right) , 1 \right),
\left( \left(\begin{array}{cc} 8 \\ 0 \end{array}\right) , 1 \right) \right\}.$$

So $A$ is  right arithmetic progression with cardinality $|A|=5$ and $B$ is a left arithmetic progression with cardinality $|B| = 9$. Further, $A$ and $B$ share the same endpoints and have the same ``common difference" $q$. Explicitly computing $A \overset{\iota}{\cdot} B$, we get 11 elements, which is equal to $|A| + |B| -3$ as Conjecture \ref{IEHconj} predicts. 


$$ A \overset{\iota}{\cdot} B= \left\{
\left( \left(\begin{array}{cc} 2 \\ 8 \end{array}\right) , 4 \right),
\left( \left(\begin{array}{cc} 4 \\ 8 \end{array}\right) , 4 \right),
\left( \left(\begin{array}{cc} 6 \\ 8 \end{array}\right) , 4\right),
\left( \left(\begin{array}{cc} 8 \\ 8 \end{array}\right) , 4\right),
\left( \left(\begin{array}{cc} 10 \\ 8 \end{array}\right) , 4\right), \right.$$

$$  \left.
\left( \left(\begin{array}{cc} 12 \\ 8 \end{array}\right) , 4\right),
\left( \left(\begin{array}{cc} 14 \\ 8 \end{array}\right) , 4 \right),
\left( \left(\begin{array}{cc} 16 \\ 8 \end{array}\right) , 4\right),
\left( \left(\begin{array}{cc} 18 \\ 8 \end{array}\right) , 4 \right),
\left( \left(\begin{array}{cc} 20 \\ 8 \end{array}\right) , 4 \right),
\left( \left(\begin{array}{cc} 22 \\ 8 \end{array}\right) , 4 \right) \right\}.$$

In the following two sections, we prove a series of results that prove the forward direction of this conjecture under the assumption that we have a priori knowledge that $A$ and $B$ are arithmetic progressions with the same common difference. 

\section{An Extension to $\Z/n\Z$ for the Inverse Theorem of the Dias da Silva-Hamidoune Theorem}

Our first result extends Theorem \ref{invDH} in $\Z/n\Z$ for composite $n$ by assuming we have a priori knowledge of the sets $A,B$ as arithmetic progressions with the same common difference, and characterizing when such $A,B$ form a critical pair, i.e. reach the lower bound of the Erd\H{o}s-Heilbronn Problem. 

\begin{theorem} Let $A,B \subseteq G=\mathbb{Z}/n\mathbb{Z}$ where $|A|=k, |B|=l$ and $p(G)$ is the smallest prime dividing n. Suppose $p(G) > k+ l - 3$ where $k,l \geq 3.$ Further suppose that $A,B$ are arithmetic progressions with the same common difference. Then we have that:
$$ |A\overset{\cdot}{+} B| = |A| + |B| -3 \text{ implies } A = B.$$
\end{theorem}

\begin{proof}
Let $d$ be the common difference of the arithmetic progressions $A$ and $B$, i.e. $A = \{ a + sd \hspace{0.2cm}| \hspace{0.2cm} 0 \leq s \leq k-1\}$ and $B = \{b + t d \hspace{0.2cm} | \hspace{0.2cm} 0 \leq t \leq l-1\}.$ Without loss of generality, we can suppose that $|A|\geq |B|,$ i.e. $ k \geq l$.

We have that
$$ A\overset{\cdot}{+} B = \{ a + sd + b + td \hspace{0.2cm}|\hspace{0.2cm} 0 \leq s \leq k-1 ,\hspace{0.2cm} 0\leq t \leq l-1,\hspace{0.2cm} a+sd\neq b + td\}.$$
Since
\begin{equation}  a+ sd + b + td = a + (s \pm 1)d+ b + (t \mp 1)d, \end{equation}
 then even if $a + sd = b + td$, by the above we have that the sum can still be written as the sum of two distinct elements from $A$ and $B$ unless
\begin{itemize}
\item[] (i) $s=t=0$ \hspace{0.15cm} \text{or}
\item[] (ii) $ s= k-1$ and $t= l-1.$
\end{itemize}

 In other words, we can find another pair of elements, $a + (s \pm 1)d \in A$ , and $b + (t \mp 1)d \in B$, that yield the same sum, unless the term in question is a shared endpoint of the arithmetic progression where $s,t = 0$ corresponds to the first endpoint and $s = k-1, t = l-1$ corresponds to the last endpoint.

Thus $A \overset{\cdot}{+} B  = A + B$ ($\Rightarrow |A \overset{\cdot}{+} B| \geq |A|+|B|-1$, contrary to our assumption), unless:
\begin{itemize}
\item[] (i)  $a=b$ \hspace{0.15cm} \text{or}
\item[] (ii) $a+(k-1)d = b + (l-1)d$
\end{itemize}
Notice without loss of generality that (ii) can be reduced to (i) by putting $\bar{a} = a + (k-1)d,  \bar{b} = b+(l-1)d$ and forming the arithmetic progressions by setting $\bar{d} = - d$.

Now since $a=b$, we have that $$ A\overset{\cdot}{+} B = \{ a + sd + a + td \hspace{0.2cm}| \hspace{0.2cm} sd \neq td\}.$$  Notice if $sd=td$ for $s\neq t$ (say without loss of generality that $s > t$), then $(s-t)d \equiv 0$ in $\mathbb{Z}/n\mathbb{Z}$, which implies that $n| (s-t)d$ . Because $d\not \equiv 0$, there is a prime $p_0$ dividing $n$ such that $p_0|(s-t)$. By our definition of $p(G)$ as the smallest prime dividing $n$, we have $$ k \geq (s-t) \geq p_0 \geq p(G) > k+ l - 3 \geq k+3-3 = k $$
which is a contradiction (Note: if we had assumed $t>s$, we would have derived a similar contradiction using $l$). Thus we must conclude that $sd = td$ implies $s=t$.

Again, we now point out that if $sd = td$, then we can write $$ a +sd+ a + td = a + (s \pm 1)d + a + (t \mp 1)d$$ unless $s,t = 0$ or $s = t= k-1= l-1$ by the previous paragraph. But if $k>l$, then we only get the case $s =t = 0$ which means that $$ |A\overset{\cdot}{+} B| \geq  | A+B\setminus \{a+a\} | \geq |A|+|B| - 1 -1 = |A| + |B| - 2.$$ This is a contradiction to our assumption, so we are forced to conclude that $k = l$, and so then $A = B$. This completes the proof.
\end{proof}

\begin{corollary}
Let $A,B \subseteq \Z/n\Z$ be arithmetic progressions with the same common difference where $|A| = k, |B| = l$. Suppose $p > k+l-3$ where $k,l \geq 5$. Then we have that:
$$ |A\overset{\cdot}{+} B| = |A| + |B| -3 \text{ if and only if } A = B.$$
\end{corollary}
\begin{proof}
The forward direction is a consequence of Theorem 5.1. The converse is a special case of Theorem 4.6 where the abelian group is $\Z/n\Z$.
\end{proof}

\begin{remark} 
\end{remark}
As pointed out to us by Gy. K{\'a}rolyi, the assumption that $A$ and $B$ are arithmetic progressions can be dropped when we are in $\Z/p\Z$ where $p$ is prime to yield the following result \cite{Gyula09}:
\begin{theorem}
Let $A,B \subseteq \Z /p\Z$ be nonempty subsets such that $p \geq |A| + |B| -2$. Then $|A \overset{\cdot}{+} B | = |A|+|B| - 3$ if and only if $A = B$ and one of the following holds:
\begin{itemize}
\item[] $(i)$ $|A|=2$ or $|A| = 3$;
\item[] $(ii)$ $|A| = 4,$ and $A = \{a, a+d, c, c+d\}$;
\item[] $(iii)$ $|A| \geq 5,$ and $A$ is an arithmetic progression.
\end{itemize}
\end{theorem}
We note that the proof of this statement relies on the polynomial method of Alon, Nathanson, and Rusza, while our proof relies solely on elementary methods albeit with the additional a priori knowledge of $A,B$ being arithmetic progressions to prove the result for general $\Z/n\Z$. It is not clear to us whether the methods used to prove Theorem 5.4 can be easily applied to prove Theorem 5.1, but we also present our elementary proof to foreshadow methods used to extend this result into nonabelian groups in the next section. 

\section{A Generalization of the Inverse Theorem of the Dias da Silva-Hamidoune Theorem to Nonabelian Groups}

In this section, we extend the results of Theorem 5.1 to nonabelian groups when the automorphism $\theta = \iota$ is the identity map so that $A \overset{\iota}{\cdot} B = \{ ab| \hspace{0.1cm} a \in A, b \in B, a \neq b \}$.  
\begin{theorem} Let $A,B \subseteq G$, where $|A|=k, |B|=l$ and $p(G)$ is the smallest prime dividing the order of Gy. Suppose $p(G) > k+ l - 2$ where $k,l \geq 3.$ Further suppose that $A$ is a right geometric progression and $B$ is a left geometric progression and that they have the same common ratio. Then we have that:
$$ |A\overset{\iota}{\cdot} B| = |A| + |B| -3 \text{ implies } A \text{ and } B \text{ have the same endpoints}.$$
\end{theorem}

\begin{proof}
Let $d$ be the common ratio of the geometric progressions $A$ and $B$, i.e. $A = \{ ad^s \hspace{0.2cm}| \hspace{0.2cm} 0 \leq s \leq k-1\}$ and $B = \{d^tb \hspace{0.2cm} | \hspace{0.2cm} 0 \leq t \leq l-1\}.$ 

So we see
$$ A\overset{\iota}{\cdot} B = \{ ad^sd^tb \hspace{0.2cm}|\hspace{0.2cm} 0 \leq s \leq k-1 ,\hspace{0.2cm} 0\leq t \leq l-1,\hspace{0.2cm} ad^s\neq d^tb\}.$$

We note that $ad^sd^tb = ad^{s\pm1}d^{t\mp1}b$.

\noindent \textbf{Subcase 1:}
\\
First suppose that there is at least one pair of elements that cannot be rewritten, i.e. that there exists $s,t$ such that $ad^s=d^tb$ and $ad^{s+1}=d^{t-1}b$. Then we see that we have $d^{t-1}b=d^tbd$ which implies $b=dbd$ and similarly $ad^{s+1}=dad^s$ implies that $a=dad$.
Looking at the unrestricted productset, we see that $AB = \{ad^rb \hspace{0.2cm}|\hspace{0.2cm} 0 \leq r \leq k + l -2\}$ and since from our initial assumption $|d| \geq p(G) > k + l -2$, $|AB| = k + l -1$. So we see the only way the order of the product can achieve the lower bound is to have two or more pairs of elements that are equal. In other words, our assumption requires the existence of $s_1,t_1,s_2,t_2$ such that $ad^{s_1}=d^{t_1}b$, $ad^{s_2}=d^{t_2}b$, with $s_1 + t_1 \neq s_2 + t_2$.
Using the identities $a=dad$ and $b=dbd$ from above, we obtain $d^{(s_1+t_1)-(s_2+t_2)}=1$, from which it follows that $ |d|$ divides $(s_1+t_1)-(s_2+t_2)$. Hence
$$ k + l -2 < p(G) \leq |d| \leq (s_1+t_1)-(s_2+t_2) \leq k + l -2$$
which is a contradiction. Thus we may conclude that this case is not possible with our conditions.

\noindent \textbf{Subcase 2:}
\\
We are now reduced to the case when $ad^s = d^tb$ always implies $ad^{s\pm1} \neq d^{t\mp1}b$ for all $s,t$. Then for each instance of restriction, we can find another pair of elements, $ad^{s \pm 1} \in A$, and $d^{t \mp 1}b \in B$, that yield the same product, unless the term in question is a shared endpoint of the geometric progression where $s,t = 0$ corresponds to the first endpoint and $s = k-1, t = l-1$ corresponds to the last endpoint.

Thus $A \overset{\iota}{\cdot} B  = AB = \{ad^rb \hspace{0.2cm}|\hspace{0.2cm} 0 \leq r \leq k + l -2\}$ (and from our initial assumptions $|d| > k + l -2$), so $|A \overset{\iota}{\cdot} B| = |AB| = k + l -1$ unless $a=b$ or $ad^{k-1} = d^{l-1}b$.

Notice without loss of generality that the second case can be reduced to the first by putting $\bar{a} = ad^{k-1},  \bar{b} = d^{l-1}b$ and forming the geometric progressions by setting $\bar{d} = d^{-1}$.

So we can assume $a=b$. Now suppose $ad^{k-1} \neq d^{l-1}b$. Then $A \overset{\iota}{\cdot} B  = A B \setminus \{ab\}$ implies that $|A \overset{\iota}{\cdot} B| = k + l -2$, which contradicts the initial assumption.
\\
Hence we have shown for both subcases that we reach a contradiction. Therefore, we are forced to conclude that $A$ and $B$ must share both endpoints. 
\end{proof}
\begin{remark}\rm
As opposed to the abelian case, we note that $A$ and $B$ can share the same endpoints and still not have the same cardinalities as shown in Example \ref{iehexample}.
\end{remark}


\section{Current Progress on the Full Conjecture}

We note in this section our most recent result whose proof will be presented in another manuscript \cite{JKRW} that is the most general statement we can prove towards the conjecture:
\begin{theorem}
Let $A,B$ be subsets of a finite group $G$ such that $k = |A|, l = |B| > 10$ and $ p(G) > (2k + 2l)^{k+l},$ and let $\sigma \in $Aut(G). If $|A \overset{\sigma}{\cdot} B| = k+l - 3$, then there exist $a, q, r \in G$ such that $ \sigma(r) = q, aq^{k-1}= p^{l-1}a$ and 
$$ A = \{a, aq, \dots, aq^{k-1}\}, B = \{a, ra, \dots, r^{l-1}a\}.$$
\end{theorem}

\section{Concluding Remarks}

We have provided a survey of results concerning both direct and inverse problems related to the Cauchy-Davenport and Erd\H{o}s-Heilbronn problems. We formulated an open conjecture concerning the inverse Erd\H{o}s - Heilbronn problem in nonabelian groups and provided a nontrivial group as an example to support our formulation.  We proved an inverse theorem of the Dias da Silva-Hamidoune theorem in $\Z/n\Z$ for composite $n$ under the assumption that $A,B$ are arithmetic progressions of the same common difference. While this result may be deducible from methods used to prove Theorem 5.4 (it is not immediate to us whether this theorem can easily be generalized to $\Z/n\Z$), we present a novel proof using only elementary methods. Further, this proof foreshadows a similar argument to extend the result into nonabelian groups for the restricted product set with identity automorphism $A \overset{\iota}{\cdot} B = \{ ab\hspace{0.1cm} \vert \hspace{0.1cm} a \in A, b \in B, a \neq b \}$. 
\\
\indent Further research includes trying to settle Conjecture \ref{IEHconj}. The example of a critical pair presented in Section 4 has led to recent discoveries of other critical pairs, and we state our latest result in Section 7 that will be presented in \cite{JKRW}. The full conjecture still eludes us, and it is unclear if our elementary methods can be utilized further in this domain. We hope a promising line of attack involving the polynomial method can be developed for nonabelian groups, and that the inverse Erd\H{o}s-Heilbronn problem with arbitrary automorphism $\theta$ can be fully established in an elegant manner.

\section*{\small{Acknowledgements}}
The authors are very deeply indebted to Gyula K{\'a}rolyi for invaluable insight and discussion concerning the conjecture of the inverse Erd\H{o}s-Heilbronn problem in nonabelian groups. The authors also wish to thank Bill Layton for his advice and guidance concerning this manuscript.

\newpage
\section{Appendix }


\begin{table}[!h] 
\caption{{\small{Timeline for Cauchy-Davenport Theorem and
Erd\H{o}s-Heilbronn Problem}}}\label{timeline}
\begin{center}
\begin{tabular}{|@{\ }c@{\ }|@{\,}c@{\,}|@{\,}c@{\,}|c|} \hline Year & Contents & Person(s) & Cite \\ \hline \hline
 $1813$ & Cauchy-Davenport Theorem for $\ZZ/p\ZZ$ & Cauchy, A.L. &  \cite{Cauchy} \\ \hline
 $1935$ & Cauchy-Davenport Theorem for $\ZZ/p\ZZ$ & Davenport, H. & \cite{Davenport} \\ \hline
 $1935$ & Cauchy-Davenport Theorem for $\ZZ/m\ZZ$ & Chowla, I. & \cite{Chowla}\\ \hline
 $1947$ & acknowledged Cauchy's work & Davenport, H. & \cite{DavenportHist} \\ \hline
$1953$ & CDT extended to abelian groups & Kneser, M. & \cite{Kneser} \\ \hline
  early & developed the Erd\H{o}s-Heilbronn conjecture & Erd\H{o}s, P. & \\
 $1960$'s\! & & Heilbronn, H. & \\ \hline
 $1963$ & stated EHP at Number Theory conference & Erd\H{o}s, P. & \cite{Erdos Col. Con.} \\ \hline
 $1964$ & EHP in paper on sumsets of congruence classes & Erd\H{o}s, P. & \cite{Erdos Heilbronn} \\
 & & Heilbronn, H. & \\ \hline
 $1971$ & EHP appeared in the book & Erd\H{o}s, P. & \cite{Erdos} \\
 & {\em Some Problems in Number Theory} & & \\ \hline
 $1980$ & EHP in {\em Old and New Problems and} & Erd\H{o}s, P. & \cite{Erdos Graham}\\
        & {\em Results in Combinatorial Number Theory}& Graham, R. & \\ \hline
$1984$ & CDT for finite groups proven & Olson, J.E. & \cite{Olson}  \\ 
  & (special case of Olson's theorem) & & \\ \hline
 $1994$ & EHP proved for special case $A=B$ & Dias da Silva, J.A.\! & \cite{Dias da Silva} \\
        & & Hamidounne, Y.O.\! & \\ \hline
        & & Alon, Noga & \\
 $1995$ & EHP proved by the Polynomial Method & \! Nathanson, M.B.\! & \cite{Alon2} \\
        & & Ruzsa, I. &  \\ \hline
 $2000$ & proved inverse EHP in the asymptotic sense for $\ZZ/p\ZZ$ & Lev, V.F. & \cite{Lev1} \\ \hline
 $2003$ & inverse CDT extended to abelian groups & K{\'a}rolyi, Gy. & \cite{Gyula3} \\ \hline
 $2004$ & EHP to abelian groups for $A=B$ & K{\'a}rolyi, Gy. &\!\cite{Gyula1,Gyula2}\!\!\\ \hline
 $2004$ & EHP to groups of prime power order & K{\'a}rolyi, Gy. & \cite{Gyula2} \\ \hline
 $2005/2006$ & CDT extended to finite groups & K{\'a}rolyi, Gy., Wheeler, J.P. & \cite{Gyula3} , \cite{Wheeler}\\ 
  & (independent of Olson and each other) & & \\ \hline
 $2006$ & EHP for finite groups & Balister, P. & \cite{Balister/Wheeler}\\
  & & Wheeler, J.P. & \\ \hline
 $2009$ & improved asymptotic inverse EHP result for $\ZZ/p\ZZ$ & Vu, V. & \cite{Vu} \\
  & & Wood, Philip M. & \\ \hline
 $2012$ & submitted proof that all critical pairs in a finite & Du, S. & \cite{Du/Pan} \\
  & nilpotent group $G$ are of the form $A=B$ & Pan, H. & \\ \hline

\end{tabular}
\end{center}
\end{table}

KEY: CDT = Cauchy-Davenport Theorem, EHP = Erd\H{o}s-Heilbronn
Problem

\end{document}